\documentclass[a4paper]{amsart}
\usepackage{url}
\usepackage{amsmath}
\usepackage[makeroom]{cancel}

\usepackage{graphicx}
\usepackage[all]{xy}
\usepackage[mathscr]{eucal}
\usepackage{amsmath,amssymb,amsfonts}

\newtheorem{theorem}{Theorem}[section]
\newtheorem{lemma}[theorem]{Lemma}
\newtheorem{corollary}[theorem]{Corollary}

\newtheorem{examples}[theorem]{Examples}
\newtheorem{proposition}[theorem]{Proposition}

\usepackage{stackengine}
\usepackage{xcolor}
\usepackage[all]{xy}

\title[An embedding theorem]{Embedding Boolean ample monoids as full submonoids of Boolean inverse monoids}

\author{Mark V. Lawson}
\address{Mark V. Lawson, Department of Mathematics,
Maxwell Institute for Mathematical Sciences,
Heriot-Watt University,
Riccarton,
Edinburgh EH14 4AS,
UNITED KINGDOM}
\email{m.v.lawson@hw.ac.uk}

\begin{document}
\dedicatory{This paper is dedicated to Victoria Gould on the occasion of her birthday.}

\begin{abstract}
We show that, in certain circumstances, a Boolean ample monoid may be fully embedded into a Boolean inverse monoid
in a way that generalizes how right reversible cancellative monoids may be embedded into groups.
We use groupoids of fractions and non-commutative Stone duality to prove the result.
\end{abstract}
\maketitle

\section{Introduction}

A semigroup $S$ is said to be {\em ample} if it is equipped with two unary operations, denoted by $a \mapsto a^{\ast}$ and $a \mapsto a^{+}$,
satisfying the following axioms:
\begin{itemize}
\item[{\rm (A1).}] $(s^{\ast})^{\ast} = s^{\ast}$, and $(s^{+})^{+} = s^{+}$, and $(s^{\ast})^{+} = s^{\ast}$, and $(s^{+})^{\ast} = s^{+}$. 
\item[{\rm (A2).}] $(s^{\ast}t^{\ast})^{\ast} = s^{\ast}t^{\ast}$, and $(s^{+}t^{+})^{+} = s^{+}t^{+}$.
\item[{\rm (A3).}] $s^{\ast}t^{\ast} = t^{\ast}s^{\ast}$, and $s^{+}t^{+} = t^{+}s^{+}$.
\item[{\rm (A4).}] $ss^{\ast} = s$ and $s^{+}s = s$
\item[{\rm (A5).}] $(st)^{\ast} = (s^{\ast}t)^{\ast}$ and $(st)^{+} = (st^{+})^{+}$.
\item[{\rm (A6).}] $t^{\ast}s = s(ts)^{\ast}$ and $st^{+} = (st^{+})^{+}s$.
\item[{\rm (A7).}] If $ac = bc$ then $ac^{+} = bc^{+}$, and if $ca = cb$ then $c^{\ast}a = c^{\ast}b$.
\end{itemize}
Elements of the form $a^{\ast}$ or $a^{+}$ are called {\em projections}.
It is easy to show that in an ample semigroup the projections are precisely the idempotents;
thus the idempotents commute.

We denote the set of idempotents of a semigroup $S$ by $\mathsf{E}(S)$.
We say that $S$ is a {\em full} subsemigroup of a semigroup $T$
if $S$ is a subsemigroup of $T$ and $\mathsf{E}(T) \subseteq S$.
An embedding $\iota \colon S \rightarrow T$ is {\em full} if $\iota (S)$ is a full subsemigroup of $T$.

\begin{examples}\mbox{}
{\em 
\begin{enumerate}
\item Ample monoids with a single idempotent are precisely the cancellative monoids.
\item The full subsemigroups of inverse semigroups are ample.  
\end{enumerate}
}
\end{examples}
The question of which cancellative monoids can be embedded into groups is a difficult one.
The analogous question for ample monoids is which of these can be fully embedded into inverse monoids.
In this paper, we shall answer a special case of this question.
We return to cancellative monoids for motivation.
Let $C$ be a cancellative monoid. 
Suppose that $C$ can be embedded into a group $G$ in such a way that $C^{-1}C = G$.
Then in fact $C$ is {\em right reversible} in the sense that for all $a,b \in C$ we have that $Ca \cap Cb \neq \varnothing$.
It is well-known that every right reversible cancellative monoid $C$ can be embedded into a group $G$ in such a way that $C^{-1}C = G$ \cite{CP}.
The group $G$ is then called a {\em group of fractions}.
We shall generalize this theorem to a class of ample monoids, called Boolean ample monoids.
We shall replace groups by Boolean inverse monoids,
and groups of fractions by groupoids of fractions in the sense of \cite{GZ}.
Rather than just an embedding we shall obtain a full embedding.
Our proof will use non-commutative Stone duality;
so, it is there we shall begin.

\section{Non-commutative Stone duality}

In this section, I shall describe the duality between Boolean restriction monoids and Boolean categories.
This duality was first established in \cite{KL}.
We shall need a proof here that is closer to the classical proof of Stone duality.
The details can be found in \cite{K2025} so we shall just sketch out the theory here;
one difference from \cite{K2025} is that we shall use prime filters and not germs.
We shall be working with posets  (the order will always be the natural partial order).
If $(P,\leq)$ is s poset and $X \subseteq P$ then 
$$X^{\uparrow} = \{p \in P \colon x \leq p \text{ for some } x \in X\}.$$

We begin with classical Stone duality \cite{BS}.
Let $B$ be a Boolean algebra.
A {\em filter} in $B$ is a subset $A$ closed under finite meets such that $A = A^{\uparrow}$.
A filter $A$ is {\em proper} if $0 \notin A$.
A maximal proper filter is called an {\em ultrafilter}.
A proper filter $A$ is said to be {\em prime} if $a \vee b \in A$ implies that $a \in A$ or $b \in A$.
In a Boolean algebra, prime filters are the same as ultrafilters.
If $a \in B$ denote by $X_{a}$ the set of all prime filters that contain $a$.
Put $\beta$ equal to the set of all $X_{a}$.
Denote by $\mathsf{X}(B)$ the set of prime filters of $B$.
Then $\beta$ is the basis for a topology on $\mathsf{X}(B)$,
which makes it a Boolean space --- that is a compact, Hausdorff $0$-dimensional space.
We call $\mathsf{X}(B)$ the {\em Stone space} of $B$.
On the other hand, given a Boolean space $X$, the set of clopen subsets $\mathsf{B}(X)$ forms a Boolean algebra.
The essence of classical Stone duality is that $B \cong \mathsf{B}(\mathsf{X}(B))$ and $X \cong \mathsf{X}(\mathsf{B}(X))$.

We now turn to the generalization of classical Stone duality.
We begin by describing the two categories that are in duality.

A {\em restriction monoid} $S$ satisfies the axioms (A1)--(A6) above.
The set of projections of $S$ is denoted by $\mathsf{Proj}(S)$;
all projections are idempotents, but not every idempotent is a projection.
If $S$ is a restriction monoid, define the natural partial order $a \leq b$ iff $a = ba^{\ast}$.
Say that $a$ and $b$ are {\em compatible}, denoted by $a \sim b$, iff $ab^{\ast} = ba^{\ast}$ and $a^{+}b = b^{+}a$.
A {\em Boolean restriction monoid} is a restriction monoid in which every pair of compatible elements has a join,
multiplication distributes over such joins (from both sides),
and the set of projections forms a Boolean algebra under the natural partial order.
A {\em homomorphism} of Boolean restriction monoids is a monoid homomorphism of restriction monoids that maps zero to zero
and preserves binary joins.
Restriction monoids and their homomorphisms form a category which we shall denote by $\mathscr{R}$.

Apart from categories of structures, all categories will be small and treated as algebraic generalizations of monoids;
thus we shall identify objects with identities.
If $C$ is a category then its set of identities will be denoted by $C_{o}$.
If $a \in C$ we denote the unique right identity of $a$ by $\mathbf{d}(a)$
and the unique left identity  by $\mathbf{r}(a)$.
The product $\exists ab$ precisely when
$\mathbf{d}(a) = \mathbf{r}(b)$. 
The partial multiplication in a category is denoted by $\mathbf{m}$. 
If the category $C$ is a subcategory of the category $D$ so that $C_{o} = D_{o}$,
then we say that $C$ is a {\em wide} subcategory of $D$.
A {\em topological category} $C$ is one equipped with a topology so that the maps $\mathbf{d}$, $\mathbf{r}$ and $\mathbf{m}$
are continuous.
A {\em cancellative category} is one in which if $ab = ac$ is defined then $b = c$,
and if $ba = ca$ is defined then $b = c$.
A {\em groupoid} is a category such that for every element $g$ there is a (perforce) unique element $g^{-1}$
such that $g^{-1}g = \mathbf{d}(g)$ and $gg^{-1} = \mathbf{r}(g)$. 
A topological category is said to be {\em \'etale} if the maps $\mathbf{d}$ and $\mathbf{r}$ are local homeomorphisms.
A {\em Boolean category} is an \'etale topological category whose space of identities is a Boolean space.
We shall need maps a bit different from functors between Boolean categories.
Let $C$ and $D$ be categories.
A binary relation $\rho \subseteq C \times D$, which we shall regard as an arrow from $D$ to $C$, is called a 
{\em relational functor}\footnote{It is not a functor.} if it satisfies the following three conditions:
\begin{itemize}
\item[{\rm (RF1).}] For each $e \in D_{o}$ there is a unique $f \in C_{o}$ such that $(f,e) \in \rho$.
\item[{\rm (RF2).}] If $(c,d) \in \rho$ then $(\mathbf{d}(c), \mathbf{d}(d)) \in \rho$ and $(\mathbf{r}(c), \mathbf{r}(d)) \in \rho$. 
\item[{\rm (RF3).}] If $(a,b), (c,d) \in \rho$ and $\exists ac$ and $\exists  bd$ then $(ac,bd) \in \rho$.
\end{itemize}
We say that such a relational functor is a {\em covering, relational functor} if the following two conditions hold:
\begin{itemize}
\item[{\rm (CRF1).}] If $(a,b), (a,b') \in \rho$ and $\mathbf{d}(b) = \mathbf{d}(b')$ then $b = b'$, and dually.
\item[{\rm (CRF2).}] If $e$ and $f$ are identities such that $(f,e) \in \rho$ and $\mathbf{d}(a) = f$ then there exists $b$ such that
$\mathbf{d}(b) = e$ and $(a,b) \in \rho$, and dually.
\end{itemize}
If $C$ and $D$ are topological categories and $\rho$ is a relational functor from $D$ to $C$,
then we say that $\rho$ is {\em continuous} if for every open set $U$ in $C$ the set 
$$\rho^{-1}(U) = \{d \in D \colon \exists c \in U, \text{ such that } (c,d) \in \rho \}$$
is open.
We say that $\rho$ is {\em proper} if $U$ compact implies that  $\rho^{-1}(U)$ is compact. 
Boolean categories together with the proper, continuous, covering, relational functors form a category which we shall denote by $\mathscr{C}$.
We have the following theorem.

\begin{theorem}\label{them:ruth} 
The category $\mathscr{R}$ is dually equivalent to the category $\mathscr{C}$.
\end{theorem}

We shall now sketch out how this theorem is proved.

We begin by defining the contravariant functor from  $\mathscr{R}$ to $\mathscr{C}$.
We describe it on objects first.
Let $S$ be a Boolean restriction monoid.
A {\em filter} in $S$ is a subset $A$ such that $A = A^{\uparrow}$ and for all $a,a' \in A$ there exists $a'' \in A$
such that $a'' \leq a',a$.
The filter $A$ is said to be {\em proper} if $0 \notin A$.
A maximal proper filter is called an {\em ultrafilter}.
A proper filter $A$ is said to be {\em prime} if $a \vee b \in A$ implies $a \in A$ or $b \in A$.
Just as with Boolean algebras, prime filters and ultrafilters are the same thing.
Let $\mathsf{C}(S)$ be the set of prime filters of $S$.
If $X$ is any subset define $X^{\ast} = \{a^{\ast} \colon a \in A\}$ and $A^{+} = \{a^{+} \colon a \in A\}$.
Define $\mathbf{d}(A) = (A^{\ast})^{\uparrow}$ and $\mathbf{r}(A) = (A^{+})^{\uparrow}$.
If $A$ is a prime filter then $\mathbf{d}(A)$ and $\mathbf{r}(A)$ 
are both prime filters.
If $A$ and $B$ are both prime filters and $\mathbf{d}(A) = \mathbf{r}(B)$ 
define
$$A \cdot B = (AB)^{\uparrow},$$ 
also a prime filter.
Then with respect to the partially defined operation $\cdot$, the set $\mathsf{C}(S)$ is a category.
The identities of this category are precisely the prime filters containing projections.
A remarkable feature of prime filters is that they are essentially `cosets' of prime filters that contain projections:
if $A$ is a prime filter then 
$$A = (a\mathbf{d}(A))^{\uparrow},$$ 
where $a \in A$.
This result immediately implies that if $A$ and $B$ are (prime) filters such that $\mathbf{d}(A) = \mathbf{d}(B)$
and $A \cap B \neq \varnothing$ then $A = B$, and dually.
We can relate prime filters in $S$ with prime filters in $\mathsf{Proj}(S)$:
if $E$ is a prime filter in $\mathsf{Proj}(S)$ then $A = (aE)^{\uparrow}$, where $a^{\ast} \in E$, is a prime filter in $S$ containing $a$
and every prime filter in $S$ is of this form, and dually.
We now need to equip the category $\mathsf{C}(S)$ with a topology.
For each $a \in S$, let $X_{a}$ be the set of all prime filters in $S$ that contain $a$.
Let $\beta$ be the set of all such $X_{a}$.
Then $\beta$ is the basis for a topology on $\mathsf{C}(S)$,
with respect to which, $\mathsf{C}(S)$ is an \'etale topological category.
The space $\mathsf{C}(S)_{o}$ is homeomorphic to the Stone space of $\mathsf{Proj}(S)$.
It follows that $\mathsf{C}(S)$ is a Boolean category.
We now need to describe the effect of our contravariant functor on homomorphisms.
Let $\theta \colon S \rightarrow T$ be a homomorphism between Boolean restriction monoids
and let $A$ be a prime filter in $T$.
It could happen that $\theta^{-1}(A)$ is empty.
If it is not empty, then Ganna Kudryavtseva discovered that $\theta^{-1}(A)$
is a disjoint union of prime filters: the easiest way to see this is to define the relation $\approx$ on  
$\theta^{-1}(A)$ by $a \approx b$ iff there is $c \in \theta^{-1}(A)$ such that $c \leq a,b$.
This is an equivalence relation and the equivalence classes are prime filters.
Define the arrow $\rho_{\theta}$ from  $\mathsf{C}(T)$ to $\mathsf{C}(S)$ by $(A,B) \in \rho_{\theta}$ iff
$A \subseteq \theta^{-1}(B)$  where $B \in \mathsf{C}(T)$.
Then $\rho_{\theta}$ is a proper, continuous, covering, relational functor.
Our contravariant functor therefore does the following: on objects $S$ it does $S \mapsto \mathsf{C}(S)$
and on homomorphisms $\theta \colon S \rightarrow T$ it does $\theta \mapsto \rho_{\theta}$.

We now define the contravariant functor from  $\mathscr{C}$ to $\mathscr{R}$.
We describe it on objects first.
Let $C$ be a Boolean category.
A subset $X \subseteq C$ is called a {\em local bisection}
if $a,b \in X$ and $\mathbf{d}(a) = \mathbf{d}(b)$ implies that $a = b$, and dually.
Under subset multiplication, if $A$ and $B$ are compact-open local bisections
then $AB$ is a compact-open local bisection.
Denote by $\mathsf{KB}(C)$ the set of all compact-open local bisections on $C$.
If $A \in \mathsf{KB}(C)$ define $A^{\ast} = \{\mathbf{d}(a) \colon a \in A\}$
and $A^{+} = \{\mathbf{r}(a) \colon a \in A\}$.
Then $\mathsf{KB}(C)$ is a Boolean restriction monoid whose set of projections is
just the set of all clopen subsets of the Boolean space $C_{o}$.
Before we consider the effect of our contravariant functor on relational functors,
we shall need the following result.

\begin{theorem}\label{them:epic}\mbox{}
\begin{enumerate}

\item Let $S$ be a Boolean restriction monoid.
Then there is an isomorphism $S \cong \mathsf{KB}(\mathsf{C}(S))$ given by the map
$a \mapsto X_{a}$.

\item Let $C$ be a Boolean category.
Then there is an isomorphism $C \cong \mathsf{C}(\mathsf{KB}(C))$ of Boolean categories
given by the map $x \mapsto F_{x}$, where $F_{x}$ is the set of all compact-open local bisections
containing $x$.

\end{enumerate}
\end{theorem}

Let $\rho \colon \mathsf{C}(T) \rightarrow \mathsf{C}(S)$ be 
a proper, continuous, covering, relational functor.
A corresponding homomorphism $\theta_{\rho} \colon \mathsf{KB}(\mathsf{C}(S)) \rightarrow \mathsf{KB}(\mathsf{C}(T))$
is defined by
$\rho^{-1}(X_{s}) = X_{\theta_{\rho}(s)}$.
(This set is compact and open since $\rho$ is proper and continuous.
The set $\rho^{-1}(X_{s})$ is a local bisection,
since its elements are prime filters.)
Our contravariant functor therefore does the following: on objects $C$ it does $C \mapsto \mathsf{KB}(C)$
and on a proper, continuous, covering, relational functor
$\rho \colon C \rightarrow B$ it does $\rho \mapsto \theta_{\rho}$.
That we really have a contravariant functor in this direction  depends on the following result, which is easy to prove.

\begin{lemma} Let $\rho$ be a covering relational functor.
If $(ab,d) \in \rho$ then there exists $d_{1}$ and $d_{2}$ such that $d = d_{1}d_{2}$
and $(a,d_{1}),(b,d_{2}) \in \rho$.
\end{lemma}

Parts (1) and (2) of the following result are proved in \cite{KL} but we give a direct proof here.

\begin{proposition}\label{prop:fury} \mbox{}
\begin{enumerate}

\item If $S$ is a Boolean ample monoid then $\mathsf{C}(S)$ is cancellative.

\item If $C$ is a Boolean cancellative category then
then $\mathsf{KB}(C)$ is ample.

\item If $S$ is a Boolean inverse monoid then $\mathsf{C}(S)$ is a groupoid.

\item If $C$ is a Boolean groupoid then
then $\mathsf{KB}(C)$ is inverse.

\end{enumerate}
\end{proposition}
\begin{proof} (1) Suppose that $A \cdot B = A \cdot C$ where $A,B,C$ are prime filters.
We prove that $B = C$. The dual result follows by symmetry.
In fact, we shall prove $B \subseteq C$.
The reverse inclusion then follows by symmetry.
Let $b \in B$.
Then $b^{+} \in \mathbf{r}(B)$.
So, $b^{+} \in \mathbf{d}(A)$, since $A \cdot B$ is defined in the category.
There is therefore $a \in A$ such that $a^{\ast} \leq b^{+}$.
But $ab \in A \cdot B$ and so $ab \in A \cdot C$.
Therefore $a_{1}c \leq ab$ where $a_{1} \in A$ and $c \in C$.
Since $a,a_{1} \in A$ there exists $a_{2} \leq a,a_{1}$ where $a_{2} \in A$.
Thus $a_{2}c \leq ab$.
But $a_{2} = aa_{2}^{\ast}$.
Thus $aa_{2}^{\ast}c = ab(a_{2}c)^{\ast}$.
We are working in an ample semigroup,
so $a^{\ast}a_{2}^{\ast}c = a^{\ast}b(a_{2}c)^{\ast}$.
Now, $a^{\ast}, a_{2}^{\ast} \in \mathbf{r}(C)$.
Thus $a^{\ast}a_{2}^{\ast}c \in \mathbf{r}(C)\mathbf{r}(C)C \subseteq C$.
It follows that  $a^{\ast}b(a_{2}c)^{\ast} \in C$.
Thus $b \in C$, as required.

(2) It is enough to check that Axiom (A7) holds.
Suppose that $AB = AC$ where $A,B,C$ are compact-open local bisections.
We prove that $A^{\ast}B = A^{\ast}C$.
Let $\mathbf{d}(a)b \in A^{\ast}B$ where $a \in A$.
Then $ab \in AB$.
By assumption, $ab = a_{1}c_{1}$ where $a_{1} \in A$ and $c_{1} \in C$.
Clearly, $\mathbf{r}(a) = \mathbf{r}(a_{1})$.
But $A$ is a local bisection.
Thus $a = a_{1}$.
We therefore have that $ab = ac_{1}$.
The category $C$ is cancellative, thus $b = c_{1}$.
It follows that $\mathbf{d}(a)b \in A^{\ast}C$, as required.

The proofs of parts (3) and (4) may be found in \cite[Chapter 9]{Lawson2025}. 
\end{proof}

\section{Etale groupoids of fractions}

We say that a category $C$ is {\em right reversible} if for all elements $a,b \in C$ such that $\mathbf{d}(a) = \mathbf{d}(b)$
we have that $Ca \cap Cb \neq \varnothing$;
we are using the same terminology as \cite{CP} since we regard small categories as generalizations of monoids.
The following result is proved in \cite{GZ}. 

\begin{theorem}\label{them:four}
Let $C$ be a right reversible cancellative category.
Then there is a groupoid $G$ such that
$C$ is embedded in $G$ by a functor $\iota$ such that
$\iota (C)$ is a wide subcategory of $G$
and 
$\iota (C)^{-1}\iota (C) = G$.
The groupoid $G$ is unique with these properties.
\end{theorem}

The groupoid $G$ is called a {\em groupoid of fractions} of $C$.
For the remainder of this section, to ease notation we shall always regard $C$ as a subcategory of $G$.



Having dealt with matters algebraical, we now deal with matters topological.

The first result is immediate because $\mathbf{d}$ and $\mathbf{r}$ are local homeomorphisms.

\begin{lemma}
In any \'etale category, the space of identities is open.
\end{lemma}

Our next result tells us that we can choose the basis of an \'etale category to have a special form.
We shall refine this result later when we deal with Boolean categories.

\begin{lemma}\label{lem:basis-open-local-bisections} 
Let $C$ be an \'etale category.
Then the open local bisections form a basis for the topology on $C$.
\end{lemma}
\begin{proof} Let $U$ be any non-empty open set in $C$ and let $a \in U$.
Then there is an open neighbourhood $V$ of $a$ on which $\mathbf{d}$ is a homeomorphism.
Likewise, there is an open neighbourhood $W$ of $a$ on which $\mathbf{r}$ is a homeomorphism.
The set $U \cap V \cap W$ is an open set containing $a$ inside $U$ which is a local bsection.
This proves that the open local bisections form a basis for $C$.
\end{proof}

If $C$ is a category and $X \subseteq C$, then we write $\mathbf{d}(X) = \{\mathbf{d}(x) \colon x \in X \}$, and dually.

\begin{lemma}\label{lem:star-compact} Let $C$ be an  \'etale topological category.
Then the open local bisection $A$ is compact iff $\mathbf{d}(A)$ is compact, and dually.
\end{lemma}
\begin{proof} Suppose that $A$ is an open local section and that $\mathbf{d}(A)$ is compact.
We prove that $A$ is compact.
Suppose that $A \subseteq \bigcup_{i \in I}U_{i}$ where $U_{i}$ are open sets of $C$.
Then $A = \bigcup_{i \in I}U_{i} \cap A$.
The sets $U_{i} \cap A$ are open and they are local bisections since they are contained in a local bisection.
It follows that $\mathbf{d}(A) =  \bigcup_{i \in I}\mathbf{d}(U_{i} \cap A)$.
By assumption, $\mathbf{d}(A)$ is compact.
We may therefore write $\mathbf{d}(A) =  \bigcup_{j = 1}^{m} \mathbf{d}(U_{j} \cap A)$, relabelling if necessary.
We claim that  $A \subseteq \bigcup_{i = 1}^{m}U_{i}$. 
Let $x \in A$.
Then $\mathbf{d}(x) \in \mathbf{d}(A)$
and so $\mathbf{d}(x) \in \mathbf{d}(U_{j} \cap A)$ for some $j$.
But $U_{j} \cap A$ is a local section.
Thus there is a unique element $y \in U_{j} \cap A$ such that $\mathbf{d}(y) = \mathbf{d}(x)$.
But $x,y \in A$ which is a local bisection.
It follows that $x = y$ and so $x \in U_{j} \cap A \subseteq U_{j}$.
The proof of the converse follows from the fact that $C$ is \'etale
\end{proof}

We may now prove the following result about Boolean categories; this is the refinement we promised earlier.

\begin{proposition}\label{prop:basis-etale} Let $C$ be a Boolean category.
Then the compact-open local bisections form a basis for the topology on $C$.
\end{proposition}
\begin{proof} By Lemma~\ref{lem:basis-open-local-bisections}, 
it is enough to consider the case where $U$ is an open local bisection in $C$.
Let $x \in U$.
There is an open set $V$ containing $x$ such that $V$ and $\mathbf{d}(V)$ are homeomorphic by $\mathbf{d}$.
Thus $U \cap V$ contains $a$ and is homeomorphic to $\mathbf{d}(U \cap V)$ under $\mathbf{d}$.
But the clopen subsets of the space of identities form a basis for the topology on the space of identities.
Thus there is an open set $Y \subseteq U \cap V$ containing $x$ such that $\mathbf{d}(Y)$ is clopen.
But the closed subsets of Hausdorff spaces are compact and so $Y$ is compact by Lemma~\ref{lem:star-compact}. 
\end{proof}

We can now state the main theorem we prove in this section;
we shall prove it later.

\begin{theorem}\label{them:stor} Let $C$ be a right reversible Boolean cancellative category.
Then its groupoid of fractions is a Boolean groupoid.
In addition, if $\beta$ is the basis for $C$ consisting of all compact-open local bisections,
then $\beta^{-1}\beta$ is a basis for the topology on the groupoid of fractions
and consists of compact-open local bisections.
\end{theorem}

We now work towards proving the above theorem.

\begin{lemma}\label{lem:kingdoms} 
Let $C$ be an \'etale topological cancellative right reversible category with groupoid of fractions $G$,
where $\beta$ is the set of all open local bisections of $C$.
Then $\beta^{-1} \beta$ is a basis for a topology on $G$.
In addition, each element of $\beta^{-1}\beta$ is a local bisection.
\end{lemma}
\begin{proof} We regard $C$ as a subcategory of $G$ in what follows to ease notation.
We shall use the fact that a basis $\beta$ for the topology in $C$ consists of open local bisections.

An element $g$ of $G$ can be written $a^{-1}b$ where $a,b \in C$.
By assumption, $a \in U \in \beta$ and $b \in V \in \beta$ for some $U$ and $V$.
It follows that $g \in U^{-1}V$.
Thus, every element of $G$ is in some element of $\beta^{-1} \beta$.

Suppose that $g \in U_{1}^{-1}V_{1} \cap U_{2}^{-1}V_{2}$, where $U_{1}, U_{2}, V_{1}, V_{2} \in \beta$.
Then $g = a_{1}^{-1}b_{1} = a_{2}^{-1}b_{2}$ where $a_{1} \in U_{1}$, $a_{2} \in U_{2}$, $b_{1} \in V_{1}$ and $b_{2} \in V_{2}$.
Observe that $\mathbf{d}(b_{1}) = \mathbf{d}(b_{2})$ and $\mathbf{d}(a_{1}) = \mathbf{d}(a_{2})$.
Since $C$ is right reversible, we may find elements $a,a' \in C$ such that $ab_{1} = a'b_{2} = r$ (say)
and $aa_{1} = a'a_{2} = s$ (say).
(That the elements $a,a'$ are the same in each case follows from the two different ways that $g$ can be written).
Choose $A,B \in \beta$ such that $A$ contains $a$ and $B$ contains $a'$.
Using the fact that the product of open local bisections is an open local bisection,
we have that 
$AV_{1} \cap BV_{2} = X$ is an open local bisection containing $r$
and
$AU_{1} \cap BU_{2} = Y$ is an open local bisection containing $s$.
Thus $X,Y \in \beta$.
Observe that $s^{-1}r \in Y^{-1}X$, which simplifies to $g$.
We have shown that $g \in Y^{-1}X$.
Now, let $h \in Y^{-1}X$ be any element.
Then we can write $h = (au)^{-1}(a_{1}v)$ where $a \in A$, $u \in U_{1}$, $a_{1} \in A$ and $v \in V_{1}$.
Thus $h = u^{-1}a^{-1}a_{1}v$.
But $a,a_{1} \in A$, a local bisection, and $\mathbf{r}(a) = \mathbf{r}(a_{1})$.
It follows that $a = a_{1}$.
We deduce that $h \in U_{1}^{-1}V_{1}$.
We may similarly show that $h \in U_{2}^{-1}V_{2}$.
We have therefore proved that $g \in Y^{-1}X \subseteq U_{1}^{-1}V_{1} \cap U_{2}^{-1}V_{2}$.

We have therefore proved that $\beta^{-1}\beta$ really is a basis.

We now prove that each element of $\beta^{-1}\beta$ is a local bisection.
Let $g,h \in U^{-1}V$ such that $\mathbf{d}(g) = \mathbf{d}(h)$.
We may write $g = u_{1}^{-1}v_{1}$ and $h = u_{2}^{-1}v_{2}$.
Thus $\mathbf{d}(v_{1}) = \mathbf{d}(v_{2})$.
It follows that $v_{1} = v_{2}$.
From this we deduce that $\mathbf{r}(u_{1}) = \mathbf{r}(u_{2})$ and so $u_{1} = u_{2}$.
It follows that $g = h$.
Symmetry delivers the result.
\end{proof}

The following result reassures us that we do get an \'etale groupoid.

\begin{proposition}\label{prop:main} Let $C$ be an \'etale topological cancellative right reversible category whose groupoid of fractions is $G$. 
Let $\beta$ be a basis for the topology on $C$ consisting of all open local bisections of $C$.
Then $G$ is an \'etale groupoid with respect to the topology with basis $\beta^{-1} \beta$.
In addition, $C$ is an open subset of $G$.
\end{proposition}
\begin{proof} 
We proved in Lemma~\ref{lem:kingdoms}, that $\beta^{-1} \beta$ is a basis for a topology on $G$. 
But the topology on $G$ has as a basis local bisections.
From this, it follows that $G$ will be \'etale.
It is clear that inversion is a homeomorphism.
It remains to show that the partial multiplication in $G$ is continuous.
To that end, 
let $g_{1}h_{1} \in U^{-1}V$ be any element where $U,V \in \beta$.
Let $g_{1} = a_{1}^{-1}b_{1}$ and $h_{1} = a_{2}^{-1}b_{2}$ where $a_{1},b_{1},a_{2},b_{2} \in C$.
By right reversibilty, 
let $p$ and $q$ be such that $pb_{1} = qa_{2} = r$.
Then $g_{1}h_{1} = (pa_{1})^{-1}(qb_{2})$.
By assumption, $g_{1}h_{1} = u^{-1}v$ where $u \in U$ and $v \in V$.
From the fact that we are dealing with a groupoid of fractions, 
it follows that there are elements $a,b \in C$ such that
$apa_{1} = bu$ and $aqb_{2} = bv$.
Let $Z$ and $B$ be open local bisections of $C$ such that 
$apb_{1} \in Z$ and $b \in B$.
Consider, first, the set $(BU)^{-1}Z$.
We have that $apa_{1} \in BU$.
Thus $a_{1}^{-1}p^{-1}a^{-1} \in (BU)^{-1}$.
By definition, $apb_{1} \in Z$.
It follows that $g_{1} \in (BU)^{-1}Z$.
Now, we consider, $Z^{-1}BV$.
By definition, $apb_{1} \in Z$.
But $pb_{1} = qa_{2}$.
Thus $aqa_{2} \in Z$.
Whence $a_{2}^{-1}q^{-1}a^{-1} \in Z^{-1}$.
But $aqb_{2} \in BV$.
It follows that $h_{1} \in Z^{-1}BV$.
We now calculate the product $[(BU)^{-1}Z][Z^{-1}BV]$.
But this is contained in $U^{-1}V$
because if $X$ is a local bisection
both $XX^{-1}$ and $X^{-1}X$ consist of identities.
This is enough to show that the partial multiplication in the groupoid $G$ is continuous. 

It only remains to show that $C$ is an open subset of $G$.
The set of identities $C_{o}$ is an open set
in $C$. It is therefore an open local bisection.
Thus $\beta \subseteq \beta \beta^{-1}$.
This means that $C$ is an open subset of its groupoid of fractions.
\end{proof}

{\em We can now prove Theorem~\ref{them:stor}.}
Let $C$ be a Boolean cancellative right reversible category whose groupoid of fractions is $G$. 
We have seen that $G$ is an \'etale category.
It is Boolean because $C$ is a wide subcategory of $G$.
Because $C$ is open in $G$, it follows that a compact subset of $C$ is also compact in $G$.
Let $\beta$ be a basis for the topology on $C$ consisting of all compact-open local bisections of $C$.
The elements of $\beta^{-1}\beta$ are compact-open local bisections of $G$.
It only remains to show that $\beta^{-1}\beta$ is a basis for the topology on $G$.
We have proved that a basis for the topology on $G$ consists of the subsets of the form
$U^{-1}V$ where $U$ and $V$ are open local bisections of $C$.
However, $\beta$ is a basis for the topology on $C$.
Thus $U$ is a union of elements of $\beta$.
Likewise, $V$ is a union of elements of $\beta$.
It follows that $U^{-1}V$ is a union of elements of $\beta^{-1}\beta$.
This proves that $\beta^{-1}\beta$ is a basis for the topology on $G$. \\

The following is a converse to what we have found.

\begin{theorem}\label{them:twelve} Let $G$ be an \'etale topological groupoid.
Suppose that $C$ is an open, wide subcategory of $G$ which is right reversible 
and such that $CC^{-1} = G$.
Then $C$ is an \'etale topological cancellative category,
and if $\alpha$ is a basis of $C$ consisting of open local bisections
then $\alpha \alpha^{-1}$ is a basis for the topology on $G$.
\end{theorem}
\begin{proof} It is immediate that $C$ is cancellative and, because $C$ is open in $G$,
it inherits the main features of the topology of $G$.
Thus $C$ is an \'etale topological category.
Let $X$ be an open local bisection of $G$ containing the element $g$.
Let $g = a^{-1}b$ where $a,b \in C$.
Let $A$ be any open local bisection of $C$ that contains $a$.
Then $b \in AX$.
Thus $Y = AX \cap C$ is non-empty, since it contains $b$, and is an open local bisection of $C$.
Thus $g \in A^{-1}Y$.
We prove that $A^{-1}Y \subseteq X$.
Let $h \in A^{-1}Y$.
Then $h = a_{1}^{-1}a_{2}x$, where $a_{1},a_{2} \in A$ and $x \in X$.
But $A$ is a local bisection and so $a_{1} = a_{2}$.
It follows that $h \in X$.
We have proved that each open local bisection of $G$ is a union of open local
bisections of the form $U^{-1}V$ where $U$ and $V$ are open local bisections of $C$.
\end{proof}

\section{The embedding theorem}

Let $S$ be an ample monoid.
From Section~2, its associated category of prime filters $\mathsf{C}(S)$ is cancellative.
The first question therefore is: what condition on $S$ is equivalent to $\mathsf{C}(S)$ being right reversible?
This is answered by the following theorem.
To state it succinctly, we introduce some notation.
Let $F$ be a prime filter in $\mathsf{E}(S)$.
We write $x \leq_{F} y$ iff $x^{\ast}, y^{\ast} \in F$ and $x \leq y$;
we say that this is the natural partial order {\em relativized to $F$}. 
Let $S$ be an ample monoid.
Then $S$ satisfies {\em condition (C)} if the following holds:
for any prime filter $F \subseteq \mathsf{E}(S)$ and elements $a, b \in S$
such that $a^{\ast}b^{\ast} \in F$ there exist non-zero elements $c,d \in S$ such that 
the following two conditions hold:
\begin{enumerate}  
\item There exists $a_{1} \leq_{F} a$ such that $a_{1}^{+} \leq c^{\ast}$;
there exists $b_{1} \leq_{F} b$ such that $b_{1}^{+} \leq d^{\ast}$.
\item There is an element $e \in F$ such that $cae = dbe$. 
\end{enumerate}

\begin{proposition}\label{prop:mary-anne} Let $S$ be an ample monoid.
Then $\mathsf{C}(S)$ is right reversible iff condition (C) holds for $S$.
\end{proposition}
\begin{proof} Suppose first that $\mathsf{C}(S)$ is right reversible.
We show that (C) holds.
Let $F \subseteq \mathsf{E}(S)$ be a prime filter and let $a, b \in S$
such that $a^{\ast}b^{\ast} \in F$. 
Put $A = (aF)^{\uparrow}$ and $B = (bF)^{\uparrow}$.
Both $A$ and $B$ are prime filters with $\mathbf{d}(A) = F^{\uparrow} = \mathbf{d}(B)$.
Thus by right reversibility, there are prime filters $C$ and $D$ such that $\exists C \cdot A$ and $\exists B \cdot D$ and $C \cdot A = D \cdot B$.
We may write $C = (c \mathbf{r}(A))^{\uparrow}$, where $c^{\ast} \in \mathbf{r}(A)$,
and $D = (d \mathbf{r}(B))^{\uparrow}$, where $d^{\ast} \in \mathbf{r}(B)$.
Then $\mathbf{r}(A) = ((aF)^{+})^{\uparrow}$ and $\mathbf{r}(B) = ((bF)^{+})^{\uparrow}$.
It follows that $(ae)^{+} \leq c^{\ast}$ for some $e \in F$
and $(bf)^{+} \leq d^{\ast}$ for some $f \in F$.
We may therefore put $a_{1} = ae$ and $b_{1} = bf$.
We can write $C \cdot A = (caF)^{\uparrow}$ and $D \cdot B = (dbF)^{\uparrow}$.
Since $\mathbf{d}(C \cdot A) = \mathbf{d}(D \cdot B)$,
the equality $C \cdot A = D \cdot B$ follows from the fact that $C \cdot A \cap D \cdot B \neq \varnothing$.
This means that $cae = dbe$ for some $e \in F$. 

We now prove the converse.
Suppose that (C) holds.
We prove that $\mathsf{C}(S)$ is right reversible.
Let $A$ and $B$ be prime filters such that $\mathbf{d}(A) = \mathbf{d}(B)$.
Then $A = (aF)^{\uparrow}$, where $a^{\ast} \in F$,
and $B = (bF)^{\uparrow}$, where $b^{\ast} \in F$.
We now invoke (C).
There are elements $c$ and $d$ having various properties.
Put $C = (c\mathbf{r}(A))^{\uparrow}$ and $D = (d\mathbf{r}(B))^{\uparrow}$.
Now, $\mathbf{r}(A) = ((aF)^{+})^{\uparrow}$.
But $a_{1} = ae$ where $e \in F$.
Thus $a_{1}^{+} \in  \mathbf{r}(A)$.
It follows that  $c^{\ast} \in \mathbf{r}(A)$.
Thus $c \in C$.
Similarly, $d \in D$.
Thus $C \cdot A = (ca \mathbf{d}(A))^{\uparrow}$ and $D \cdot B = (db\mathbf{d}(A))$.
There is some $e \in F$ such that $cae = dbe$.
It follows that $C \cdot A \cap D \cdot B \neq \varnothing$.
Thus, in fact, $C \cdot A = D \cdot B$, and so $\mathsf{C}(S)$ is right reversible. 
\end{proof}

Condition (C) is not particularly nice, but it does have a pleasant consequence.

\begin{corollary} Let $S$ be a Boolean ample monoid satisfying (C).
Then for all elements $a$ and $b$ such that $a^{\ast}b^{\ast} \neq 0$
we have that $Sa \cap Sb \neq \{0\}$.
\end{corollary}
\begin{proof} Because $a^{\ast}b^{\ast} \neq 0$, there is a prime filter $F$ in $\mathsf{E}(S)$ containing $a^{\ast}b^{\ast}$.
Invoking condition (C), there are elements $c$ and $d$ and an idempotent $e \in F$ such that
$cae = dbe$.
We now apply Axiom (A6), to get that
$(c(ae)^{+})a = (d(be)^{+})b$ which is non-zero since it lives in a prime filter.
\end{proof}

The condition (C) does have a natural interpretation.

\begin{proposition}\label{prop:homuz} Let $S$ be a full submonoid of a Boolean inverse monoid $T$.
We suppose also that $S$ is a Boolean ample monoid.
Then $\mathsf{C}(T)$ is a groupoid of fractions of $\mathsf{C}(S)$ iff
each non-zero element of $T$ can be written as a finite join of elements of the form $a^{-1}b$
where $a,b \in S$.
\end{proposition}
\begin{proof} We define first a functor $\iota \colon \mathsf{C}(S) \rightarrow \mathsf{C}(T)$
which is an embedding where $\iota (\mathsf{C}(S))$ is a wide subcategory of $\mathsf{C}(T)$.
This will depend only on the fact that $S$ is a full submonoid of $T$.
This will involve comparing prime filters in $S$ with prime filters in $T$.
Let $A$ be a prime filter of $S$.
Then 
$$A^{u} = \{t \in T \colon a \leq t \text{ for some } a \in A\}$$ 
is a prime filter in $T$ that contains elements of $S$.
Conversely, let $B$ be a prime filter in $T$ that contains elements of $S$.
Then 
$$B^{d} = B \cap S$$
is a prime filter in $S$.
Observe that $(A^{u})^{d} = A$ and $(B^{d})^{u} = B$.
We have therefore defined a bijection between the prime filters in $S$ and the prime filters in $T$ that contain elements of $S$.
We therefore define the functor $\iota$ by $A \mapsto A^{u}$.
We may now prove the proposition.

Suppose that each non-zero element of $T$ can be written as a finite join of elements of the form $a^{-1}b$ where $a,b \in S$.
Let $X$ be any prime filter of $T$.
Let $t \in X$ be any element, necessarily non-zero.
Then from the fact that we are dealing with a prime filter and given how $t$ can be written, there exists $a^{-1}b \in X$ for some $a,b \in S$.
Then $X = A \cdot B$ where $a^{-1} \in A$ and $b \in B$.
We may therefore  write $X = ((A)^{-1})^{-1} \cdot B$
where
$A = (a^{-1}\mathbf{r}(B))^{\uparrow}$
and
$B =(b \mathbf{d}(X))^{\uparrow}$.
Observe that $a \in A^{-1}$
and so both $A^{-1}$ and $B$ contain elements of $S$.
It follows that both $A^{-1}$ and $B$ are in the image of $\iota$.

To prove the converse,
let $t$ be any non-zero element of $T$.
By assumption, any prime filter $X$ of $T$ containing $t$ must have the form $A^{-1} \cdot B$ where $A$ and $B$ are 
prime filters of $T$ that have non-empty intersections with $S$.
Let $a \in A \cap S$ and let $b \in B \cap S$.
Thus $a^{-1}b \in A^{-1}  \cdot B = X$.
Now,  $X$ is a prime filter that contains both $t$ and $a^{-1}b$.
But any element less than $a^{-1}b$ has the form $a_{1}^{-1}b_{1}$ where $a_{1}, b_{1} \in S$;
this follows from the fact that $S$ is an order ideal of $T$.
Whence $X_{t} = \bigcup_{a^{-1}b \leq t, a,b \in S} X_{a^{-1}b}$.
By compactness $t = \bigvee_{i=1}^{n}a_{i}^{-1}b_{i}$, where $a_{i},b_{i} \in S$.
\end{proof}

We can now prove our main theorem.
We use non-commuative Stone duality on objects only;
we could use the full version of non-commutative Stone duality, but it is easier to construct embeddings directly.

\begin{theorem}\label{them:vicky}
Let $S$ be a Boolean ample monoid that satisfies condition (C).
Then $S$ can be embedded as a full submonoid of a Boolean inverse monoid $T$
in such a way that every non-zero element of $T$ is a finite
join of elements of the form $a^{-1}b$ where $a,b \in S$.
\end{theorem}
\begin{proof} Let $S$ be a Boolean ample monoid satisfying condition (C).
By Theorem~\ref{them:ruth} and part (1) of Proposition~\ref{prop:fury}, the category $\mathsf{C}(S)$ is cancellative.
By Proposition~\ref{prop:mary-anne}, the cancellative category $\mathsf{C}(S)$ is right reversible because condition (C) holds.
Thus by Theorem~\ref{them:four}, there is a groupoid $G$ which contains $\mathsf{C}(S)$ as a wide subcategory.
However, by Theorem~\ref{them:stor}, $G$ is, in fact, a Boolean groupoid.
Thus $\mathsf{C}(S) \subseteq G$ as an open and wide subcategory;
we proved openness in Proposition~\ref{prop:main}.
It follows that every compact-open local bisection of $\mathsf{C}(S)$ is a compact-open local bisection of $G$.
Thus $\mathsf{KB}(\mathsf{C}(S))$ is a full submonoid of $\mathsf{KB}(G)$; 
the fact that we have a full submonoid is a result of the fact that the category $\mathsf{C}(S)$
is wide in $G$.
By part (1) of Theorem~\ref{them:epic}, $S$ is isomorphic to $\mathsf{KB}(\mathsf{C}(S))$, and we identify these two monoids.
By part (4) of Proposition~\ref{prop:fury},
$\mathsf{KB}(G)$ is a Boolean inverse monoid $T$, say.
We have therefore shown that $S$ is a full submonoid of the Boolean inverse monoid $T$.
Finally, by Proposition~\ref{prop:homuz}, every  non-zero element of $T$ is a finite
join of elements of the form $a^{-1}b$ where $a,b \in S$.
\end{proof}

Etale groupoids of fractions  arise naturally in the theory of $C^{\ast}$-algebras \cite{RW}.


\end{document}